\documentclass[sort&compress,3p,final]{elsarticle}

\usepackage{lineno,hyperref}
\modulolinenumbers[5]

\journal{Journal of Statistical Planning and Inference}

\usepackage{enumerate,amsmath,amssymb,dsfont,eucal,multirow,bm,booktabs,amsthm}
\usepackage[justification=centering]{caption}

\newtheorem{theorem}{Theorem}[section]
\newtheorem{lemma}[theorem]{Lemma}

\newtheorem{algorithm}[theorem]{Algorithm}
\newdefinition{remark}[theorem]{Remark}

\renewcommand{\P}{\mathbf P}
\newcommand{\E}{\mathbf E}
\newcommand{\R}{\mathbb R}
\newcommand{\N}{\mathbb N}
\newcommand{\ind}{\mathbf1}
\newcommand*{\abs}[1]{\left|#1\right|}
\newcommand{\eps}{\varepsilon}
\newcommand{\F}{\mathcal F}

\newcommand*{\set}[1]{\left\{#1\right\}}
\DeclareMathOperator{\var}{Var}

\makeatletter
\def\ps@pprintTitle{%
 \let\@oddhead\@empty
 \let\@evenhead\@empty
 \def\@oddfoot{}%
 \let\@evenfoot\@oddfoot}
\makeatother
\begin{document}
\begin{frontmatter}

\title{Hypothesis testing of the drift parameter sign for fractional Ornstein--Uhlenbeck process}

\author[addressone]{Alexander Kukush}
\ead{alexander\_kukush@univ.kiev.ua}
\author[addresstwo]{Yuliya Mishura}
\ead{myus@univ.kiev.ua}
\author[addresstwo]{Kostiantyn Ralchenko}
\ead{k.ralchenko@gmail.com}
\address[addressone]{Department of Mathematical Analysis, Taras Shevchenko National University of Kyiv,\\
64 Volodymyrska, 01601 Kyiv, Ukraine}
\address[addresstwo]{Department of Probability Theory, Statistics and Actuarial Mathematics,\\ Taras Shevchenko National University of Kyiv,
64 Volodymyrska, 01601 Kyiv, Ukraine}

\begin{abstract}
We consider the fractional Ornstein--Uhlenbeck process with an unknown drift parameter and known Hurst parameter $H$.
We propose a new method to test the hypothesis of the sign of the parameter and prove the consistency of the test.
Contrary to the previous works, our approach is applicable for all $H\in(0,1)$.
We also study the estimators for drift parameter for continuous and discrete observations and prove their strong consistency for all $H\in(0,1)$.
\end{abstract}


\end{frontmatter}

%\linenumbers

\section{Introduction}

Let $(\Omega,\F,\P)$ be a probability space, and
$B^H=\set{B^H_t, t\in\R}$
be a fractional Brownian motion with Hurst parameter $H\in(0,1)$ on this probability space, that is a  centered Gaussian process with covariance function
\[
\E B^H_tB^H_s=\frac12\left(\abs{t}^{2H}+\abs{s}^{2H}-\abs{t-s}^{2H}\right),
\quad t,s\in\R.
\]
Since
$\E\left(B^H_t-B^H_s\right)^2=\abs{t-s}^{2H}$
and the process $B^H$ is Gaussian,
it has a continuous modification by Kolmogorov's theorem.
In what follows we consider such modification.

The present paper deals with the inference problem associated with the Langevin equation
\begin{equation}\label{eq:sde}
X_t=x_0+\theta\int_0^tX_s\,ds+B^H_t,
\quad t\ge0,
\end{equation}
where $x_0\in\R$, and $\theta\in\R$ is an unknown drift parameter.
This equation has a unique solution, see~\cite{CKM}.
In order to avoid integration with respect to the fractional Brownian motion for $0<H<1/2$, we can write this solution in the following form
\begin{equation}\label{eq:solution}
X_t=x_0e^{\theta t}+\theta e^{\theta t}\int_0^t e^{-\theta s}B^H_s\,ds+B^H_t,
\quad t\ge0.
\end{equation}
The process $X=\set{X_t, t\ge0}$ is called a \emph{fractional Ornstein--Uhlenbeck process}~\cite{CKM}.
It is a Gaussian process, consequently  its one-dimen\-sional distributions   are normal,  with mean
$x_0e^{\theta t}$
and variance
\[
v(\theta,t)=H\int_0^ts^{2H-1}\left(e^{\theta s}+e^{\theta(2t-s)}\right)ds,
\]
see Lemma~\ref{l:distr} in Appendix.

 The estimation problem for the drift parameter $\theta$ in the model~\eqref{eq:sde} was studied in many works.
 We refer to the paper~\cite{kumirase} for the extended survey of these results.
 The MLE's were studied in \cite{KB, Tanaka13, Tanaka15} for $H\geq 1/2$ and in the paper \cite{tudorviens} for $H<1/2$.
 Note that the MLE is hardly discretized because it contains the stochastic integrals with singular kernels.
 Therefore, several nonstandard estimators have been proposed recently.
 In particular,  for $\theta<0$ (the ergodic case) and $H\geq 1/2$  Hu and Nualart \cite{HN} constructed the analog of the least-squares estimator of the form
 \begin{equation}\label{eq:LSE}
 \hat{\theta}_T=\frac{\int_0^TX_tdX_t}{\int_0^TX^2_tdt},
 \end{equation}
where the integral $\int_0^TX_tdX_t$ is the divergence-type one.
 As an alternative, they considered  the estimator
 \begin{equation}\label{eq:HuNu}
\hat\theta_T=-\left(\frac{1}{H\Gamma(2H)T}\int_0^TX_t^2dt\right)^{-\frac{1}{2H}}.
 \end{equation}
In the non-ergodic case, when $\theta>0$,  Belfadli et al.~\cite{BESO} proposed the estimator
 \begin{equation}\label{eq:Belfadli}
\hat\theta_T
=\frac{X_T^2}{2\int_0^TX_t^2dt}
 \end{equation}
and proved its strong consistency for $H>1/2$.
Note that in this case the estimator \eqref{eq:Belfadli} coincides with the estimator \eqref{eq:LSE}, where the integral $\int_0^TX_tdX_t$ is understood in the path-wise sense.
In the papers  \cite{CES,Es-Sebaiy,ESN,HuSong13,kumirase,xiaoZX} the discretized estimators were proposed.

The methods of constructing the estimators and their asymptotic properties  essentially depend on the sign of unknown drift parameter $\theta$.
In particular, the estimator~\eqref{eq:HuNu} is based on the ergodicity and does not work in the non-ergodic case. Similarly, the estimator~\eqref{eq:Belfadli} converges to zero if $\theta<0$, see remark at the end of Sec.~3 in \cite{HN}.

The cases $H > 1/2$ (long-range dependence) and $H < 1/2$ (short-range dependence) also differ substantially.
For now, the case $H\ge1/2$ has been deeply investigated.
But we can mention only few papers devoted to the inference problem in the case $H<1/2$ (\cite{kumirase,tudorviens}).
 However, the observations of the real financial markets demonstrate that the Hurst index often falls below the level of $1/2$, taking values around 0.45--0.49 (\cite{multi}).
 If the trajectory of the process is observed continuously, then the case $H<1/2$ can be reduced to the case $H\ge1/2$ by the integral transformation of Jost~\cite[Cor.~5.2]{Jost06}. Moreover, recently Tanaka~\cite{Tanaka15} has shown that the distributions of the MLE for $H$ and $1-H$ coincide. Note that these results are not applicable in the case of discrete-time observations. Mention also that the discretized estimator proposed in \cite{kumirase} for $H<1/2$, in reality works properly only for $\theta\geq 0$ and   apparently does not work in   the ergodic case. In general, the problem of the discretization for $H<1/2$ and $\theta<0$ is open.

The above discussion motivates the hypothesis testing of the sign of drift parameter in the model~\eqref{eq:sde}.
The interest to this problem is also connected with the stability properties of the solution to the equation~\eqref{eq:sde}, which also depend on the sign of $\theta$.
For $H\ge1/2$, this problem was studied by Moers~\cite{Moers12}.
He constructed a test using the estimator
\begin{equation}\label{eq:est-Moers}
\tilde\theta_{T,H}=\frac{X_T^2-X_0^2}{2\int_0^TX_t^2dt}
-\left(\frac{1}{H\Gamma(2H)T}\int_0^TX_t^2dt\right)^{-\frac{1}{2H}}.
\end{equation}
The exact distribution of $\tilde\theta_{T,H}$ is not known, and the test is based on the asymptotic distribution of $T\tilde\theta_{T,H}$. The values of the corresponding test statistic should be compared with quantiles of the random variable
\begin{equation}\label{eq:distr-Moers}
\frac{\left(B^H_1\right)^2}{2\int_0^1\left(B^H_t\right)^2dt}
-\left(\frac{1}{H\Gamma(2H)}\int_0^1\left(B^H_t\right)^2dt\right)^{-\frac{1}{2H}},
\end{equation}
and the quantiles can be obtained by Monte Carlo simulation. The test can be used for the testing three types of hypothesis: $H_0\colon\theta\ge0$ against $H_1\colon\theta<0$, $H_0\colon\theta\le0$ against $H_1\colon\theta>0$,
and $H_0\colon\theta=0$ against $H_1\colon\theta\ne0$.
The consistency of the test is proved only for $H\in[1/2,3/4)$ for a simple alternative $\theta_1<0$, and for $H\in[1/2,1)$ for $\theta_1>0$.
Tanaka~\cite{Tanaka13,Tanaka15} considered  the testing of the hypothesis $H_0\colon\theta=0$ against the alternatives $H_1\colon\theta<0$ and $H_1\colon\theta>0$. He proposed tests based on the MLE (for both alternatives) and on the minimum contrast estimator (only for the ergodic case). Those tests were considered also for $H\ge1/2$.
To the best of our knowledge, there are no tests, suitable for the discrete-time observations of the process.

In the present paper we propose comparatively simple test for testing the null hypothesis $H_0\colon\theta\le0$ against the alternative $H_1\colon\theta>0$.
The main advantage of our approach is that it can be used for any $H\in(0,1)$. Moreover, the test is based on the observation of the process $X$ at one point, therefore, it is applicable for both continuous and discrete cases.
The distribution of the test statistic is computed explicitly, and the power of test can be found numerically for any given simple alternative.
Also we consider the hypothesis testing $H_0\colon\theta\ge\theta_0$ against $H_1\colon\theta\le0$, where $\theta_0\in(0,1)$ is some fixed number.
Unfortunately, our approach does not enable to test the hypothesis $H_0\colon\theta=0$ against the two-sided alternative $H_1\colon\theta\ne0$.

The second goal of the paper is to propose the estimators for the unknown drift parameter, which can be applied for all $H\in(0,1)$.
First, we prove that estimators~\eqref{eq:HuNu} for the ergodic case and~\eqref{eq:Belfadli} for the non-ergodic one are strongly consistent for all $H\in(0,1)$, not only for $H\ge1/2$. Then we consider the discretized versions of these estimators and prove their strong consistency.
Thus our results generalize the corresponding strong consistency results from~\cite{BESO,CES,Es-Sebaiy,ESN,HN,HuSong13,xiaoZX} for the case of arbitrary $H$.

The paper is organized as follows.
In Section~\ref{sec:Hyp} the problem of hypothesis testing for the sign of the drift parameter $\theta$ is considered.
In Section~\ref{sec:Est} we study the strong consistency of estimators for $\theta$.
Section~\ref{sec:Num} is devoted to numerics.
In Appendix we get some auxiliary results. In particular, we calculate the first two moments  of the fractional Ornstein--Uhlenbeck process.

\section{Hypothesis testing of the  drift parameter sign}\label{sec:Hyp}

\subsection{Test statistic}
For hypothesis testing of the sign of the parameter $\theta$ we construct a test based on the asymptotic behavior of the random variable
\begin{equation}\label{eq:Z}
Z(t)=\frac{\log^+\log\abs{X_t}}{\log t}, t>1.
\end{equation}
The following result explains the main idea. It is based on the different asymptotic behavior of the fractional Ornstein--Uhlenbeck process with positive drift parameter and negative one.
\begin{lemma}
The value of  $Z(t)$
converges a.\,s.\ to 1 for $\theta>0$, and to 0 for $\theta\le0$, as $t\to\infty$.
\end{lemma}
\begin{proof}
For $\theta>0$, Lemma~\ref{l:as_non-erg} implies the convergence
\[
\log\abs{X_t}-\theta t\to\log\abs{\xi_\theta}
\quad\text{a.\,s., as }t\to\infty,
\]
where $\xi_\theta$ is a normal random variable, hence, $0<\abs{\xi_\theta}<\infty$ a.\,s.
Therefore,
\[
\frac{\log\abs{X_t}}{t}\to\theta
\quad\text{a.\,s., as }t\to\infty.
\]
It means that there exists $\Omega'\subset\Omega$ such that $\P(\Omega')=1$ and for any $\omega\in\Omega'$ there exists $t(\omega)$ such that for  $t\geq t(\omega)$ $\log\abs{X_t}>0$.
Hence, for   $t\geq t(\omega)$ we have that
\[
\abs{\frac{\log^+\log\abs{X_t}}{\log t}-1}=\abs{\frac{\log\log\abs{X_t}}{\log t}-1}
=\abs{\frac{\log\log\abs{X_t}-\log t}{\log t}}
=\abs{\frac{\log\frac{\log\abs{X_t}}{t}}{\log t}}
\to0
\]
 a.\,s.,  as $t\to\infty$.
For $\theta\le0$, it follows from~\eqref{eq:est_KMRS_neg} that
\begin{equation*}\begin{gathered}
\abs{Z(t)}
\le\abs{\frac{\log^+\left(\log\left(1+t^H\log^2t\right)+\log\zeta\right)}{\log t}}=\abs{\frac{\log\left(\log\left(1+t^H\log^2t\right)+\log\zeta\right)}{\log t}}
\\ \sim\abs{\frac{\log\left(\log\left(t^H\log^2t\right)\right)}{\log t}}
\to0,\end{gathered}\end{equation*} as $t\to\infty.$
\end{proof}

The next result gives the cdf of $Z(t)$.
Let $\Phi$ and $\varphi$ denote the cdf and pdf, respectively,  of the standard normal variable.

\begin{lemma}\label{l:sup}
For $t>1$ the probability
$g(\theta,x_0,t,c)=\P(Z(t)\le c)$
is given by
\begin{equation}\label{eq:g}
g(\theta,x_0,t,c)=\Phi\left(\frac{e^{t^c}-x_0e^{\theta t}}{\sqrt{v(\theta,t)}}\right)+\Phi\left(\frac{e^{t^c}+x_0e^{\theta t}}{\sqrt{v(\theta,t)}}\right)-1,
\end{equation}
and $g$ is a decreasing function of $\theta\in\R$.
\end{lemma}
\begin{proof}
Using Lemma~\ref{l:distr} and taking into account that $\log^+x$ is a non-de\-creas\-ing function, we obtain
\begin{align*}
\P(Z(t)\le c)&=\P\left(\abs{X_t}\le e^{t^c}\right)
=\Phi\left(\frac{e^{t^c}-x_0e^{\theta t}}{\sqrt{v(\theta,t)}}\right)-\Phi\left(\frac{-e^{t^c}-x_0e^{\theta t}}{\sqrt{v(\theta,t)}}\right)\\
&=\Phi\left(\frac{e^{t^c}-x_0e^{\theta t}}{\sqrt{v(\theta,t)}}\right)+\Phi\left(\frac{e^{t^c}+x_0e^{\theta t}}{\sqrt{v(\theta,t)}}\right)-1.
\end{align*}
Let us prove the monotonicity of the function $g$ with respect to $\theta$.
Note that $g$ is an even function with respect to $x_0$.
Therefore, it suffices to consider only the case $x_0\ge0$.
The partial derivative equals
\begin{equation}\label{eq:deriv_g}
\begin{split}
\frac{\partial g}{\partial \theta}
&=\varphi\left(\frac{e^{t^c}-x_0e^{\theta t}}{\sqrt{v(\theta,t)}}\right)\left(-x_0te^{\theta t}v^{-\frac12}(\theta,t)-\frac12v^{-\frac32}(\theta,t)v'_\theta(\theta,t)\left(e^{t^c}-x_0e^{\theta t}\right)\right)\\
&\quad+\varphi\left(\frac{e^{t^c}+x_0e^{\theta t}}{\sqrt{v(\theta,t)}}\right)\left(x_0te^{\theta t}v^{-\frac12}(\theta,t)-\frac12v^{-\frac32}(\theta,t)v'_\theta(\theta,t)\left(e^{t^c}+x_0e^{\theta t}\right)\right)\\
&=-\frac12v^{-\frac32}(\theta,t)v'_\theta(\theta,t)e^{t^c}
\left(\varphi\left(\frac{e^{t^c}-x_0e^{\theta t}}{\sqrt{v(\theta,t)}}\right)+\varphi\left(\frac{e^{t^c}+x_0e^{\theta t}}{\sqrt{v(\theta,t)}}\right)\right)\\
&\quad-x_0e^{\theta t}v^{-\frac32}(\theta,t) \left(tv(\theta,t)-\frac12v'_\theta(\theta,t)\right)
\left(\varphi\left(\frac{e^{t^c}-x_0e^{\theta t}}{\sqrt{v(\theta,t)}}\right)-\varphi\left(\frac{e^{t^c}+x_0e^{\theta t}}{\sqrt{v(\theta,t)}}\right)\right).
\end{split}
\end{equation}
Since
\begin{equation}\label{deriv_v}
v'_\theta(\theta,t)=H\int_0^ts^{2H-1}\left(se^{\theta s}+(2t-s)e^{\theta(2t-s)}\right)ds>0,
\end{equation}
we see that the first term in the left-hand side of~\eqref{eq:deriv_g} is negative.
Let us consider the second term.
From~\eqref{eq:v} and \eqref{deriv_v} it follows that
\[
tv(\theta,t)-\tfrac12v'_\theta(\theta,t)=
H\int_0^ts^{2H-1}\left(\left(t-\tfrac12s\right)e^{\theta s}+\tfrac12se^{\theta(2t-s)}\right)ds>0.
\]
Since
$\abs{e^{t^c}-x_0e^{\theta t}}\le e^{t^c}+x_0e^{\theta t}$
for $x_0\ge0$, we have
\[\varphi\left(\frac{e^{t^c}-x_0e^{\theta t}}{\sqrt{v(\theta,t)}}\right)
-\varphi\left(\frac{e^{t^c}+x_0e^{\theta_1 t}}{\sqrt{v(\theta,t)}}\right)
\ge0.
\]
Thus, the second term in the left-hand side of~\eqref{eq:deriv_g} is non-positive.
Hence, $\frac{\partial g}{\partial\theta}<0$.
\end{proof}

\subsection{Testing the hypothesis \texorpdfstring{$H_0\colon\theta\le0$}{H0: theta<=0} against \texorpdfstring{$H_1\colon\theta>0$}{H1: theta>0}}
We consider the test with the following procedure of
testing the hypothesis $H_0\colon\theta\le0$ against the alternative $H_1\colon\theta>0$.
For a given significance level $\alpha$, and for sufficiently large value of $t$ we choose a threshold $c=c_t\in(0,1)$, see Lemma~\ref{l:c_t}.
Further, when $Z(t)\le c$ the hypothesis $H_0$ is accepted, and when $Z(t)>c$ it is rejected.
Below we will propose a technically simpler version of this test, without the computation of $c$, see Remark~\ref{rem:1}. The threshold $c$ can be chosen as follows.

Fix a number $\alpha\in(0,1)$, the significance level of the test. This level gives the maximal probability of a type I error, that is in our case the probability to reject the hypothesis $H_0\colon\theta\le0$ when it is true.
By Lemma~\ref{l:sup}, for a threshold $c\in(0,1)$ and $t>1$ this probability equals
\[
\sup_{\theta\le0}\P(Z(t)>c)=1-g(0,x_0,t,c).
\]
Therefore, we determine $c_t$  as a solution to the equation
\begin{equation}\label{eq:typeIerror}
g(0,x_0,t,c_t)=1-\alpha.
\end{equation}
The following result shows that for any $\alpha\in(0,1)$, it is possible to choose a sufficiently large $t$ such that $c_t\in(0,1)$.
\begin{lemma}\label{l:c_t}
Let $\alpha\in(0,1)$.
Then there exists $t_0\ge1$ such that for all $t>t_0$ there exists a unique $c_t\in(0,1)$ such that
$g(0,x_0,t,c_t)=1-\alpha$.
Moreover $c_t\to0$, as $t\to\infty$.

The constant $t_0$ can be chosen as the largest $t\ge1$ that satisfies at least one of the following two equalities
\begin{equation}\label{eq:t_0}
g(0,x_0,t,0)=1-\alpha
\quad\text{or}\quad
g(0,x_0,t,1)=1-\alpha.
\end{equation}
\end{lemma}
\begin{proof}
By Lemma~\ref{l:asymp_v}~$(iii)$, $v(0,t)=t^{2H}$. Then for $\theta=0$ the formula~\eqref{eq:g} becomes
\begin{equation}\label{eq:g0}
g(0,x_0,t,c)=\Phi\left(\frac{e^{t^c}-x_0}{t^H}\right)+\Phi\left(\frac{e^{t^c}+x_0}{t^H}\right)-1.
\end{equation}
For any $t>1$, the function $g(0,x_0,t,c)$
is strictly increasing with respect to $c$.
For $c=0$ we have
\[
g(0,x_0,t,0)=\Phi\left(\frac{e-x_0}{t^H}\right)+\Phi\left(\frac{e+x_0}{t^H}\right)-1
\to2\Phi(0)-1=0,\quad\text{as } t\to\infty.
\]
Therefore, there exists $t_1>1$ such that
$g(0,x_0,t,0)<1-\alpha$
for all $t\ge t_1$.

Similarly, for $c=1$
\[
g(0,x_0,t,1)
=\Phi\left(\frac{e^{t}-x_0}{t^H}\right)+\Phi\left(\frac{e^{t}+x_0}{t^H}\right)-1
\to2\Phi(\infty)-1=1,\quad\text{as }  t\to\infty.
\]
Therefore, there exists  $t_2>1$ such that $g(0,x_0,t,1)>1-\alpha$ for all $t\ge t_2$.

Thus, for any $t\ge t_0=\max\set{t_1,t_2}$
there exists a unique $c_t\in(0,1)$ such that $g(0,x_0,t,c_t)=1-\alpha$.

To prove the convergence $c_t\to0$, $t\to\infty$, consider an arbitrary $\eps\in(0,1)$.
Then
\[
g(0,x_0,t,\eps)
=\Phi\left(\frac{e^{t^\eps}-x_0}{t^H}\right)+\Phi\left(\frac{e^{t^\eps}+x_0}{t^H}\right)-1
\to2\Phi(\infty)-1=1,\quad\text{as } t\to\infty.
\]
Arguing as above, we see that there exists $t_3>1$ such that for any $t>t_3$ the unique $c_t\in(0,1)$, for which $g(0,x_0,t,c_t)=1-\alpha$, belongs to the interval $(0,\eps)$.
This implies the convergence $c_t\to0$, as $t\to\infty$.

It follows from \eqref{eq:g0} that
$g(0,x_0,t,0)=g(0,x_0,t,1)$ for $t=1$.
As $t\to\infty$, we have $g(0,x_0,t,0)\to0$, $g(0,x_0,t,1)\to1$.
Hence, at least one of the equalities \eqref{eq:t_0} is satisfied for some $t\ge1$ and the set of such $t$'s is bounded.
\end{proof}

\begin{remark}\label{rem:1}
Since the function $g(0,x_0,t,c)$ is strictly increasing with respect to $c$ for $t>1$, we see that the inequality $Z(t)\le c_t$ is equivalent to the inequality
$g(0,x_0,t,Z(t))\le g(0,x_0,t,c_t)=1-\alpha$.
Therefore, we do not need to compute the value of $c_t$.
It is sufficient to compare $g(0,x_0,t,Z(t))$ with the level $1-\alpha$.
\end{remark}

\begin{algorithm}\label{alg:1}
The hypothesis $H_0\colon\theta\le0$ against the alternative $H_1\colon\theta>0$ can be tested as follows.
\begin{enumerate}
\item
Find $t_0$ defined in Lemma~\ref{l:c_t}. The algorithm can be applied only in the case $t>t_0$.
\item
Evaluate the statistic $Z(t)$ defined by~\eqref{eq:Z}.
\item
Compute the value of $g(0,x_0,t,Z(t))$, see~\eqref{eq:g0}.
\item
Accept the hypothesis $H_0$ if $g(0,x_0,t,Z(t))\le1-\alpha$, and the hypothesis $H_1$, otherwise.
\end{enumerate}
\end{algorithm}

\begin{remark}
In fact, the condition $t>t_0$ is not too restrictive, since for reasonable values of $\alpha$, the values of $t_0$ are quite small, see Table~\ref{tab:hyp1-t0}.
\end{remark}

Let us summarize the properties of the test in the following theorem.
\begin{theorem}
The test described in Algorithm~\ref{alg:1} is unbiased and consistent, as $t\to\infty$.
For the simple alternative $\theta_1>0$ and moment $t>t_0$, the power of the test equals $1-g(\theta_1,x_0,t,c_t)$, where $c_t$ can be found from~\eqref{eq:typeIerror}.
\end{theorem}
\begin{proof}
It follows from the monotonicity of $g$ with respect to $\theta$ (see Lemma~\ref{l:sup}) that for any $\theta_1>0$
$$
\P(Z(t)> c_t)=1-g(\theta_1,x_0,t,c_t)>1-g(0,x_0,t,c_t)=\alpha.
$$
This means that the test is unbiased.
Evidently, for the simple alternative $\theta_1>0$ the power of the test equals $1-g(\theta_1,x_0,t,c_t)$.

It follows from the convergence $c_t\to0$, as $t\to\infty$ (see Lemma~\ref{l:c_t}), that $c_t<c$ for sufficiently large $t$ and some constant $c\in(0,1)$.
Taking into account the formula~\eqref{eq:g} and Lemma~\ref{l:asymp_v}~$(i)$, we get, as $t\to\infty$:
\begin{align*}
1&\ge 1-g(\theta_1,x_0,t,c_t)
\ge 1-g(\theta_1,x_0,t,c)
=2-\Phi\left(\frac{e^{t^c}-x_0e^{\theta_1 t}}{\sqrt{v(\theta_1,t)}}\right)-\Phi\left(\frac{e^{t^c}+x_0e^{\theta_1 t}}{\sqrt{v(\theta_1,t)}}\right)\\
&\to2-\Phi\left(-\frac{x_0\theta_1^H}{\sqrt{H\Gamma(2H)}}\right)
-\Phi\left(\frac{x_0\theta_1^H}{\sqrt{H\Gamma(2H)}}\right)
=1.
\end{align*}
Hence, the test is consistent.
\end{proof}

\subsection{Testing the hypothesis \texorpdfstring{$H_0\colon\theta\ge\theta_0$}{H0: theta>=theta0} against \texorpdfstring{$H_1\colon\theta\le0$}{H1: theta<=0}}
Fix $\theta_0\in(0,1)$.
Let us consider the problem of testing the hypothesis
$H_0\colon \theta\ge\theta_0$
against alternative
$H_1\colon \theta\le0$.

\begin{algorithm}\label{alg:2}
The hypothesis $H_0\colon\theta\ge\theta_0$ against the alternative $H_1\colon\theta\le0$ can be tested as follows.
\begin{enumerate}
\item
Find $\tilde t_0$ defined in Lemma~\ref{l:hyp2-c_t}. The algorithm can be applied only in the case $t>\tilde t_0$.
\item
Evaluate the statistic $Z(t)$ defined by~\eqref{eq:Z}.
\item
Compute the value of $g(\theta_0,x_0,t,Z(t))$, see~\eqref{eq:g}.
\item
Accept the hypothesis $H_0$ if $g(\theta_0,x_0,t,Z(t))\ge\alpha$, and the hypothesis $H_1$, otherwise.
\end{enumerate}
\end{algorithm}
 This algorithm is based on the following results. They can be proved similarly to the previous subsection.
\begin{lemma}\label{l:hyp2-c_t}
Let $\alpha\in(0,1)$.
There exists $\tilde t_0>1$ such that for all $t>\tilde t_0$ there exists a unique $\tilde c_t\in(0,1)$ such that
\begin{equation}\label{eq:typeIerror-2}
g(\theta_0,x_0,t,\tilde c_t)=\alpha.
\end{equation}
In this case $\tilde c_t\to1$, as $t\to\infty$.

The constant $\tilde t_0$ can be chosen as the largest $t>1$ that satisfies at least one of the following two equalities
\[
g(\theta_0,x_0,t,0)=\alpha
\quad\text{or}\quad
g(\theta_0,x_0,t,1)=\alpha.
\]
\end{lemma}

\begin{theorem}
The test described in Algorithm~\ref{alg:2} is unbiased and consistent, as $t\to\infty$.
For the simple alternative $\theta_1\le0$ and moment $t>\tilde t_0$, the power of the test equals $g(\theta_1,x_0,t,\tilde c_t)$, where $\tilde c_t$ can be found from~\eqref{eq:typeIerror-2}.
\end{theorem}

\begin{remark}
The values of $\tilde t_0$ for various values of $\theta_0$ and $H$ are represented in Table~\ref{tab:hyp2-t0}.
We see that if $\theta_0$ is too close to zero, then for small $H$, the condition $t>\tilde t_0$ does not hold for reasonable values of $t$.
\end{remark}

\begin{remark}
If we have a confidence interval for $\theta$, then the value of $\theta_0 \in (0, 1)$ can be chosen less than or equal to a lower confidence bound (in the case when it is positive).
\end{remark}

\section{Drift parameter estimation}\label{sec:Est}

In this section we propose drift parameter estimators that work for any $H\in(0,1)$. We consider continuous and discrete observations.

\subsection{Continuous case}
Assume that a trajectory of $X=X(t)$ is observed over a finite time interval $[0,T]$.
\begin{theorem}\label{th:cont}
Let $H\in(0,1)$.
\begin{itemize}
\item[$(i)$]
For $\theta<0$, the estimator
\[
\hat\theta_T^{(1)}
=-\left(\frac{1}{H\Gamma(2H)T}\int_0^TX_t^2dt\right)^{-\frac{1}{2H}}
\]
is strongly consistent, as $T\to\infty$.
\item[$(ii)$]
For $\theta>0$, the estimator
\[
\hat\theta_T^{(2)}
=\frac{X_T^2}{2\int_0^TX_t^2dt}
\]
is strongly consistent, as $T\to\infty$.
\end{itemize}
\end{theorem}

\begin{proof}
$(i)$
For $\theta<0$ the result follows from Lemma~\ref{l:as_erg}.

$(ii)$ If $\theta>0$, then Lemma~\ref{l:as_non-erg} implies the a.\,s.\ convergence
\begin{equation}\label{eq:th1-1}
\frac{X_T^2}{e^{2\theta T}}\to\xi_\theta^2,
\quad\text{as }T\to\infty.
\end{equation}
Therefore, by L'H\^opital's rule,
\begin{equation}\label{eq:th1-2}
\lim_{T\to\infty}\frac{\int_0^TX_t^2\,dt}{e^{2\theta T}}
=\lim_{T\to\infty}\frac{X_T^2}{2\theta e^{2\theta T}}
=\frac{\xi_\theta^2}{2\theta}.
\end{equation}
Note that $0<\xi_\theta^2<\infty$  with probability 1, since  $\xi_\theta$ is a normal random variable.
Combining~\eqref{eq:th1-1} and ~\eqref{eq:th1-2}, we get the convergence $\hat\theta_T^{(2)}\to\theta$ a.\,s., as $T\to\infty$.
\end{proof}

\begin{remark}
In the case $H\in[1/2,1)$ the strong consistency of the estimators $\hat\theta_T^{(1)}$, $\hat\theta_T^{(2)}$
was proved in~\cite{HN} and \cite{BESO}, respectively.
\end{remark}

\subsection{Discrete case}
Assume that a trajectory of $X=X(t)$ is observed at the points
$t_{k,n}=\frac{k}{n}$, $0\le k\le n^m$, $n\ge1$, where $m>1$ is some fixed natural number.
\begin{theorem}
Let $H\in(0,1)$, $m>1$.
\begin{itemize}
\item[$(i)$]
For $\theta<0$, the estimator
\[
\hat\theta_n^{(3)}(m)=-\left(\frac{1}{H\Gamma(2H)n^m}\sum_{k=0}^{n^m-1}X_{k/n}^2\right)^{-\frac1{2H}}
\]
is strongly consistent, as $n\to\infty$.
\item[$(ii)$]
For $\theta>0$, the estimator
\[
\hat\theta_n^{(4)}(m)=\frac{nX_{n^{m-1}}^2}{2\sum_{k=0}^{n^m-1}X_{k/n}^2}
\]
is strongly consistent, as $n\to\infty$.
\end{itemize}
\end{theorem}

\begin{proof}
$(i)$
Taking into account Theorem~\ref{th:cont}~$(i)$,
it suffices to prove the convergence
\begin{equation}\label{eq:discr1}
\zeta_n:=\frac{1}{n^{m-1}}\int_0^{n^{m-1}}X^2_t\,dt-\frac{1}{n^m}\sum_{k=0}^{n^m-1}X_{k/n}^2
\to0\quad\text{a.\,s., as }n\to\infty.
\end{equation}
Denote
\[
Z_n(t)=\sum_{k=0}^{n^m-1}\left(X^2_t-X_{k/n}^2\right)\ind_{\left[\frac kn,\frac {k+1}n\right)}(t).
\]
Then
\[
\zeta_n=\frac{1}{n^{m-1}}\int_0^{n^{m-1}}Z_n(t)\,dt.
\]
Using Lemma~\ref{l:moments}, one can show that
\[
\E\abs{Z_n(t)}^p\le K(p) n^{-pH}
\]
for some constant $K(p)>0$.
Then by H\"older's inequality,
\[
\E\abs{\zeta_n}^p\le K(p) n^{-pH}.
\]
Therefore, by~\cite[Lemma~2.1]{Kloeden07} for all $\eps>0$, there exists a random variable $\eta_\eps$ such that
\[
\abs{\zeta_n}\le \eta_\eps n^{-H+\eps}
\quad\text{a.\,s.}
\]
for all $n\in\N$.
Moreover, $\E\abs{\eta_\eps}^p<\infty$
for all $p\ge1$.
This implies the convergence $\zeta_n\to0$
a.\,s., as $n\to\infty$.

$(ii)$
It follows from~\cite[Cor.~5.2(i)]{kumirase} that for $\theta>0$,
\[
\frac1n\sum_{k=0}^{n^m-1}X_{k/n}^2=\int_0^{n^{m-1}}X^2_t\,dt+\vartheta_n,
\]
where
\[
\frac{\vartheta_n}{e^{2\theta n^{m-1}}}\to0 \quad\text{a.\,s., as } n\to\infty.
\]
Combining this with Theorem~\ref{th:cont}~$(ii)$ and \eqref{eq:th1-1}, we get
\[
\hat\theta_n^{(4)}(m)=\frac{X_{n^{m-1}}^2}{2\int_0^{n^{m-1}}X^2_t\,dt+2\vartheta_n}
=\left(\frac{1}{\hat\theta_{n^{m-1}}^{(2)}}+2\cdot\frac{e^{2\theta n^{m-1}}}{X_{n^{m-1}}^2}\cdot\frac{\vartheta_n}{e^{2\theta n^{m-1}}}\right)^{-1}
\to\theta
\]
a.\,s., as $n\to\infty$.
\end{proof}

\allowdisplaybreaks
\section{Numerical illustrations}\label{sec:Num}

In this section we illustrate the performance of our algorithms and estimators by simulation experiments. We choose the initial value $x_0=1$ for all simulations.

In Tables \ref{tab:hyp1-t0}--\ref{tab:hyp2-t0} the values of $t_0$ and $\tilde t_0$ for various $H$ and $\theta_0$ are given.

\begin{table}
\caption{\label{tab:hyp1-t0} Value of $t_0$ for various $H$ and $\alpha$}\footnotesize
\centering
\begin{tabular}{*{10}{l}}\toprule
$H$ & $0.1$ & $0.2$ & $0.3$  & $0.4$ & $0.5$ & $0.6$ & $0.7$ & $0.8$ & $0.9$\\
\midrule
$\alpha=0.01$ & $1.2157$ & $1.2313$ & $1.2492$ & $1.2699$ & $1.2940$ & $1.3224$ & $1.3561$  & $1.3968$ & $1.4462$
\\
$\alpha=0.05$ & $1.5310$ & $1.2373$ & $1.1526$ & $1.1124$ & $1.0889$ & $1.0736$ & $1.0627$  & $1.0547$ & $1.0485$
\\\bottomrule
\end{tabular}
\end{table}

\begin{table}
\caption{\label{tab:hyp2-t0} Values of $\tilde t_0$ for various $H$ and $\theta_0$ ($x_0=1$)}
\footnotesize
\centering
\begin{tabular}{*{10}{l}}
\toprule
$H$ & $0.1$ & $0.2$ & $0.3$  & $0.4$ & $0.5$ & $0.6$ & $0.7$ & $0.8$ & $0.9$\\
\midrule
$\theta_0=0.1$ & $32.43$ & $32.67$ & $31.99$ & $30.59$ & $28.66$ & $26.38$ & $23.90$  & $21.39$ & $18.95$
\\
$\theta_0=0.05$ & $65.24$ & $64.72$ & $61.73$ & $57.08$ & $51.41$ & $45.23$ & $38.97$  & $33.00$ & $27.62$
\\
$\theta_0=0.01$ & $326.47$ & $307.43$ & $271.64$ & $227.99$ & $181.64$ & $137.06$ & $98.76$  & $69.62$ & $49.41$
\\
$\theta_0=0.001$ & $3193.6$ & $2719.1$ & $2073.5$ & $1387.8$ & $778.9$ & $382.1$ & $189.7$  & $104.1$ & $63.6$
\\
$\theta_0=0$ & $2.34\cdot10^{16}$ & $1.53\cdot10^{8}$ & $285\,900$ & $12\,364.1$ & $1878.1$ & $534.7$ & $218.0$  & $111.2$ & $65.9$
\\\bottomrule
\end{tabular}
\end{table}

We simulate fractional Brownian motion at the points $t=0,h,2h,3h,\dots$ and compute the approximate values of the Ornstein--Uhlenbeck process as the solution to the equation~\eqref{eq:sde}, using Euler's approximations.
For various values of  $\theta$ we simulate $n=1000$ sample path with the step $h=1/10000$.
Then we apply our algorithms, choosing the significance level $\alpha=0.05$.
In Table~\ref{tab:hyp1} the empirical rejection probabilities of the test of Algorithm~\ref{alg:1} for the hypothesis testing $H_0\colon\theta\le0$ against the alternative  $H_1\colon\theta>0$ for $H=0.3$ and $H=0.7$ are reported.

\begin{table}
\caption{\label{tab:hyp1}Empirical rejection probabilities of the test of Algorithm~\ref{alg:1}\\
    for the hypothesis testing $H_0\colon\theta\le0$ against the alternative  $H_1\colon\theta>0$\\
   for $H=0.3$ and $H=0.7$}
\footnotesize
\centering
\begin{tabular}{*{10}{l}}\toprule
$\theta$ & $-0.1$  & $-0.05$ & $0$ & $0.05$ & $0.1$  & $0.15$ & $0.2$ & $0.25$ & $0.3$\\
\midrule
\multicolumn{10}{c}{$\bm{H=0.3}$}\\
$t=20$ & $0.000$ & $0.003$ & $0.043$ & $0.341$ & $0.701$ & $0.880$ & $0.973$ & $0.982$ & $0.996$
\\
$t=40$ & $0.000$ & $0.000$ & $0.043$ & $0.675$ & $0.952$ & $0.995$ & $0.999$ & $1.000$ & $1.000$
\\
$t=60$ & $0.000$ & $0.000$ & $0.039$ & $0.860$ & $0.994$ & $1.000$ & $1.000$ & $1.000$ & $1.000$
\\
$t=80$ & $0.000$ & $0.000$ & $0.048$ & $0.940$ & $1.000$ & $1.000$ & $1.000$ & $1.000$ & $1.000$
\\
$t=100$ & $0.000$ & $0.000$ & $0.049$ & $0.986$ & $1.000$ & $1.000$ & $1.000$ & $1.000$ & $1.000$
\\
\addlinespace
\multicolumn{10}{c}{$\bm{H=0.7}$}\\
$t=20$ & $0.000$ & $0.001$ & $0.058$ & $0.284$ & $0.540$ & $0.800$ & $0.910$ & $0.967$ & $0.979$
\\
$t=40$ & $0.000$ & $0.000$ & $0.050$ & $0.581$ & $0.889$ & $0.984$ & $0.998$ & $1.000$ & $1.000$
\\
$t=60$ & $0.000$ & $0.000$ & $0.042$ & $0.782$ & $0.980$ & $1.000$ & $0.999$ & $1.000$ & $1.000$
\\
$t=80$ & $0.000$ & $0.000$ & $0.047$ & $0.908$ & $0.995$ & $1.000$ & $1.000$ & $1.000$ & $1.000$
\\
$t=100$ & $0.000$ & $0.000$ & $0.048$ & $0.959$ & $1.000$ & $1.000$ & $1.000$ & $1.000$ & $1.000$
\\\bottomrule
\end{tabular}
\end{table}

Then we test the same hypothesis with the help of the test of Moers~\cite{Moers12} for $H=0.7$.
By Monte Carlo simulations for 20\,000 sample paths of $\set{B^H_t,t\in[0,1]}$ we estimate the $(1-\alpha)$-quantile $\psi_{1-\alpha}$ of the distribution~\eqref{eq:distr-Moers} for $\alpha=0.05$. Then we compare the statistic $t\tilde\theta_{t,H}$ (see~\eqref{eq:est-Moers}) with the value of this quantile and reject the hypothesis $H_0\colon\theta\le0$ if $t\tilde\theta_{t,H}>\psi_{1-\alpha}$.
We obtained that $\psi_{0.95}\approx0.827946$.
The empirical rejection probabilities for this test are given in Table~\ref{tab:hyp1-Moers}.
We see that comparing to our algorithm, the test of Moers has bigger power, i.\,e., it works a bit better when the alternative is true. But for $\theta=0$, the necessary significance level $\alpha=0.05$ is not achieved.

\begin{table}
\caption{\label{tab:hyp1-Moers}Empirical rejection probabilities of the test of Moers~\cite{Moers12}\\ for the hypothesis testing $H_0\colon\theta\le0$ against the alternative  $H_1\colon\theta>0$ for $H=0.7$}
\footnotesize
\centering
\begin{tabular}{*{10}{l}}\toprule
$\theta$ & $-0.1$  & $-0.05$ & $0$ & $0.05$ & $0.1$  & $0.15$ & $0.2$ & $0.25$ & $0.3$\\
\midrule
$t=20$    & $0.001$ & $0.013$ & $0.085$ & $0.370$ & $0.706$ & $0.873$ & $0.947$ & $0.976$ & $0.992$\\
$t=40$    & $0.000$ & $0.004$ & $0.095$ & $0.682$ & $0.948$ & $0.993$ & $0.999$ & $1.000$ & $1.000$\\
$t=60$    & $0.000$ & $0.002$ & $0.092$ & $0.881$ & $0.995$ & $1.000$ & $1.000$ & $1.000$ & $1.000$\\
$t=80$    & $0.000$ & $0.000$ & $0.105$ & $0.948$ & $0.999$ & $1.000$ & $1.000$ & $1.000$ & $1.000$\\
$t=100$  & $0.000$ & $0.000$ & $0.089$ & $0.977$ & $1.000$ & $1.000$ & $1.000$ & $1.000$ & $1.000$\\
\bottomrule
\end{tabular}
\end{table}

Tables~\ref{tab:hyp2-0.1} and \ref{tab:hyp2-0.05} represent empirical rejection probabilities of the test of Algorithm~\ref{alg:2} for $\theta_0=0.1$ and $\theta_0=0.05$, respectively.

\begin{table}
\caption{\label{tab:hyp2-0.1}Empirical rejection probabilities of the test of Algorithm~\ref{alg:2}\\ for the hypothesis testing $H_0\colon\theta\ge\theta_0$ against the alternative  $H_1\colon\theta\le0$\\ for $\theta_0=0.1$, $H=0.3$ and $H=0.7$}
\footnotesize
\centering
\begin{tabular}{*{9}{l}}\toprule
$\theta$ & $-0.4$ & $-0.3$ & $-0.2$  & $-0.1$ & $0$ & $0.1$ & $0.2$ & $0.3$\\
\midrule
\multicolumn{9}{c}{$\bm{H=0.3}$}\\
$t=32$  & $0.999$   & $0.996$   & $0.985$   & $0.961$   & $0.626$   & $0.052$   & $0.001$   & $0.000$
\\
$t=33$  & $0.999$   & $0.999$   & $0.997$   & $0.978$   & $0.667$   & $0.046$   & $0.000$   & $0.000$
\\
$t=34$  & $0.999$   & $0.998$   & $0.994$   & $0.989$   & $0.721$   & $0.042$   & $0.001$   & $0.000$
\\
$t=35$  & $1.000$   & $1.000$   & $0.999$   & $0.992$   & $0.764$   & $0.047$   &  $0.000$  & $0.000$
\\
$t=36$  & $1.000$   & $1.000$   & $1.000$   & $1.000$   & $0.805$   & $0.048$   &  $0.000$  & $0.000$
\\
\addlinespace
\multicolumn{9}{c}{$\bm{H=0.7}$}\\
$t=25$ & $0.954$ & $0.936$ & $0.798$ & $0.598$ & $0.255$ & $0.030$ & $0.003$ & $0.000$
\\
$t=30$ & $0.998$ & $0.994$ & $0.959$ & $0.817$ & $0.363$ & $0.039$ & $0.003$ & $0.000$
\\
$t=35$ & $1.000$ & $1.000$ & $1.000$ & $0.969$ & $0.513$ & $0.043$ & $0.003$ & $0.000$
\\
$t=40$ & $1.000$ & $1.000$ & $1.000$ & $1.000$ & $0.732$ & $0.047$ & $0.001$ & $0.000$
\\
$t=45$ & $1.000$ & $1.000$ & $1.000$ & $1.000$ & $0.906$ & $0.046$ & $0.000$ & $0.000$
\\\bottomrule
\end{tabular}
\end{table}

\begin{table}
\caption{\label{tab:hyp2-0.05}Empirical rejection probabilities of the test of Algorithm~\ref{alg:2}\\ for the hypothesis testing $H_0\colon\theta\ge\theta_0$ against the alternative  $H_1\colon\theta\le0$\\ for $\theta_0=0.05$, $H=0.7$}
\footnotesize
\centering
\begin{tabular}{*{12}{l}}\toprule
$\theta$ & $-0.35$  & $-0.3$  & $-0.25$ & $-0.2$ & $-0.15$ & $-0.1$  & $-0.05$ & $0$ & $0.05$ & $0.1$ & $0.15$\\
\midrule
$t=40$ & $0.904$ & $0.893$ & $0.842$ & $0.773$ & $0.661$ & $0.566$ & $0.368$ & $0.149$ & $0.051$ &  $0.007$ &  $0.000$
\\
$t=50$ & $0.999$ & $0.991$ & $0.978$ & $0.948$ & $0.901$ & $0.793$ & $0.555$ & $0.209$ & $0.053$ & $0.003$ &  $0.000$
\\
$t=60$ & $1.000$  & $1.000$ & $0.999$ & $1.000$ & $0.990$ & $0.955$ & $0.799$ & $0.346$ & $0.047$ &  $0.002$ &  $0.000$
\\
$t=70$ & $1.000$  & $1.000$ & $1.000$  & $1.000$ & $1.000$ & $0.999$ & $0.964$ & $0.504$ & $0.045$ &  $0.001$ &  $0.000$
\\
$t=80$ & $1.000$ & $1.000$ & $1.000$  & $1.000$ & $1.000$ & $1.000$ & $0.999$ & $0.719$ & $0.044$ &  $0.001$ &  $0.000$
\\\bottomrule
\end{tabular}
\end{table}

In Tables~\ref{tab:est1}-\ref{tab:est2} the quality of estimators $\hat\theta_n^{(3)}(2)$ and $\hat\theta_n^{(4)}(2)$ is studied for $\theta=-1$ and  $\theta=1$, respectively. In this case we again choose $x_0=1$ and simulate $n=100$ trajectories of the Ornstein--Uhlenbeck process with the step $h=1/2000$. We see that both estimators converge to the true value of the parameter. Note that in the ergodic case, the best results are obtained for $H=0.5$. In the non-ergodic case, the behavior of the estimator does not depend on $H$ substantially and the standard deviation is close to zero.

\begin{table}
\caption{The estimator $\hat\theta_n^{(3)}(m)$
for $\theta=-1$, $m=2$}\label{tab:est1}
\footnotesize
\centering
\begin{tabular}{*{8}{l}}\toprule
\multicolumn{2}{r}{$n$}  & $10$ & $50$ & $100$ & $200$ & $500$ & $1000$\\
\midrule
\multirow{2}{*}{$H=0.1$} & Mean  & $-0.7417$ & $-0.8550$ & $-0.9308$ & $-0.9619$ & $-0.9805$ & $-0.9878$\\
& Std.\,dev.  & $0.83493$ & $0.25492$ & $0.19887$ & $0.14314$ & $0.08914$ & $0.06096$
\\\addlinespace
\multirow{2}{*}{$H=0.3$} & Mean  & $-0.9434$ & $-0.9940$ & $-0.9875$ & $-0.9913$ & $-0.9856$ & $-0.9941$\\
& Std.\,dev.  & $0.46723$ & $0.19727$ & $0.14965$ & $0.11490$ & $0.06408$ & $0.04367$
\\\addlinespace
\multirow{2}{*}{$H=0.5$} & Mean  & $-1.1299$ & $-1.0288$ & $-1.0168$ & $-1.0118$ & $-0.9990$ & $-0.9980$\\
& Std.\,dev.  & $0.50298$ & $0.23068$ & $0.15288$ & $0.10412$ & $0.06861$ & $0.04729$
\\\addlinespace
\multirow{2}{*}{$H=0.7$} & Mean  & $-1.2482$ & $-1.0634$ & $-1.0309$ & $-1.0096$ & $-0.9954$ & $-0.9963$\\
& Std.\,dev.  & $0.54527$ & $0.22771$ & $0.16644$ & $0.11165$ & $0.07714$ & $0.05332$
\\\addlinespace
\multirow{2}{*}{$H=0.9$} & Mean  & $-1.4098$ & $-1.2191$ & $-1.1654$ & $-1.1007$ & $-1.0701$ & $-1.0621$\\
& Std.\,dev.  & $0.54264$ & $0.39257$ & $0.32444$ & $0.27009$ & $0.23992$ & $0.19265$
\\\bottomrule
\end{tabular}
\end{table}

\begin{table}[h!]
\caption{The estimator $\hat\theta_n^{(4)}(m)$
for $\theta=1$, $m=2$}\label{tab:est2}
\footnotesize
\centering
\begin{tabular}{*{8}{l}}\toprule
\multicolumn{2}{r}{$n$}  & $10$ & $50$ & $100$ & $200$ & $500$ & $1000$\\
\midrule
\multirow{2}{*}{$H=0.1$} & Mean  & $1.10671$ & $1.02001$ & $1.00981$ & $1.00476$ & $1.00175$ & $1.00075$\\
& Std.\,dev.  & $1.7356\cdot10^{-4}$ & $2.3711\cdot10^{-15}$ & $2.0606\cdot10^{-15}$ & $2.4239\cdot10^{-15}$ & $0.\cdot10^{-15}$ & $0.\cdot10^{-16}$
\\\addlinespace
\multirow{2}{*}{$H=0.3$} & Mean  & $1.10673$ & $1.02001$ & $1.00981$ & $1.00476$ & $1.00175$ & $1.00075$\\
& Std.\,dev.  & $1.8528\cdot10^{-4}$ & $2.3656\cdot10^{-15}$ & $2.2042\cdot10^{-15}$ & $2.6221\cdot10^{-15}$ & $0.\cdot10^{-15}$ & $0.\cdot10^{-16}$
\\\addlinespace
\multirow{2}{*}{$H=0.5$} & Mean  & $1.10671$ & $1.02001$ & $1.00981$ & $1.00476$ & $1.00175$ & $1.00075$\\
& Std.\,dev.  & $3.5147\cdot10^{-4}$ & $2.5233\cdot10^{-15}$ & $2.3291\cdot10^{-15}$ & $2.4917\cdot10^{-15}$ & $0.\cdot10^{-15}$ & $0.\cdot10^{-16}$
\\\addlinespace
\multirow{2}{*}{$H=0.7$} & Mean  & $1.10665$ & $1.02001$ & $1.00981$ & $1.00476$ & $1.00175$ & $1.00075$\\
& Std.\,dev.  & $1.4344\cdot10^{-3}$ & $2.3801\cdot10^{-15}$ & $2.0968\cdot10^{-15}$ & $2.1345\cdot10^{-15}$ & $0.\cdot10^{-15}$ & $0.\cdot10^{-16}$
\\\addlinespace
\multirow{2}{*}{$H=0.9$} & Mean  & $1.10633$ & $1.02001$ & $1.00981$ & $1.00476$ & $1.00175$ & $1.00075$\\
& Std.\,dev.  & $4.1894\cdot10^{-3}$ & $2.2577\cdot10^{-15}$ & $2.1988\cdot10^{-15}$ & $2.4117\cdot10^{-15}$ & $0.\cdot10^{-15}$ & $0.\cdot10^{-16}$
\\\bottomrule
\end{tabular}
\end{table}

\section{Appendix}
\subsection{One-dimensional distributions of the fractional Ornstein--Uhlenbeck process}
Let $\set{X_t,t\ge1}$ be the fractional Ornstein--Uhlenbeck process defined by~\eqref{eq:sde}.
\begin{lemma}\label{l:distr}
The random variable $X_t$ has normal distribution
$\mathcal N\left(x_0e^{\theta t}\!,v(\theta,t)\right)$,
with variance
\begin{equation}\label{eq:v}
v(\theta,t)=H\int_0^ts^{2H-1}\left(e^{\theta s}+e^{\theta(2t-s)}\right)ds.
\end{equation}
\end{lemma}

\begin{proof}
Since $B^H_t$ is a centered Gaussian process, it immediately follows from \eqref{eq:solution} that $X_t$ has normal  distribution with mean
$x_0e^{\theta t}$.
Let us calculate its variance.
We have
\begin{equation}\label{var}\begin{split}
\var &X_t=\E\left(B^H_t+\theta e^{\theta t}\int_0^t e^{-\theta s}B^H_s\,ds\right)^2\\
&=\E\left[\left(B^H_t\right)^2\right] + 2\theta e^{\theta t}\int_0^te^{-\theta s} \E\left[B^H_tB^H_s\right]\,ds
+\theta^2e^{2\theta t}\int_0^t\int_0^te^{-\theta s-\theta u} \E\left[B^H_sB^H_u\right]\,ds\,du\\
& =t^{2H} + \theta e^{\theta t}\int_0^te^{-\theta s} \left(t^{2H}+s^{2H}-(t-s)^{2H}\right)\,ds\\
  &\quad+\frac{\theta^2}{2}e^{2\theta t}\int_0^t\int_0^te^{-\theta s-\theta u}\left(s^{2H}+u^{2H}-\abs{s-u}^{2H}\right)\,ds\,du\\
 &=t^{2H} +e^{\theta t}t^{2H}\left(1-e^{-\theta t}\right)
+\theta e^{\theta t}\int_0^te^{-\theta s} s^{2H}\,ds
- \theta e^{\theta t}\int_0^te^{-\theta(t-u)}u^{2H}\,du\\
 &\quad+\theta^2e^{2\theta t}\int_0^te^{-\theta s}s^{2H}\,ds\int_0^te^{-\theta u}\,du
-\frac{\theta^2}{2}e^{2\theta t}\int_0^t\int_0^te^{-\theta s-\theta u}\abs{s-u}^{2H}ds\,du\\
 &=e^{\theta t}t^{2H}
+\theta e^{\theta t}\int_0^te^{-\theta s} s^{2H}\,ds
- \theta\int_0^te^{\theta s}s^{2H}\,ds\\
&\quad+\theta e^{2\theta t}\left(1-e^{-\theta t}\right)\int_0^te^{-\theta s}s^{2H}\,ds
 -\frac{\theta^2}{2}e^{2\theta t}\int_0^t\int_0^te^{-\theta s-\theta u}\abs{s-u}^{2H}ds\,du\\
 &=e^{\theta t}t^{2H}
- \theta\int_0^te^{\theta s}s^{2H}\,ds
+\theta e^{2\theta t}\int_0^te^{-\theta s}s^{2H}\,ds
-\frac{\theta^2}{2}e^{2\theta t}\int_0^t\int_0^te^{-\theta s-\theta u}\abs{s-u}^{2H}\,ds\,du.
\end{split}\end{equation}
The last summand can be rewritten as follows
\begin{align*}
&\frac{\theta^2}{2}e^{2\theta t}\int_0^t\int_0^te^{-\theta s-\theta u}\abs{s-u}^{2H}\,ds\,du\\
&\quad=\frac{\theta^2}{2}e^{2\theta t}\left(\int_0^t\int_0^se^{-\theta s-\theta u}(s-u)^{2H}\,du\,ds+\int_0^t\int_s^te^{-\theta s-\theta u}(u-s)^{2H}\,du\,ds\right)\\
&\quad=\theta^2e^{2\theta t}\int_0^t\int_0^se^{-\theta s-\theta u}(s-u)^{2H}\,du\,ds
=\theta^2e^{2\theta t}\int_0^t\int_0^se^{-2\theta s+\theta v}v^{2H}\,dv\,ds\\
&\quad=\frac\theta2 e^{2\theta t}\int_0^te^{-\theta v}v^{2H}\,dv
-\frac\theta2 \int_0^te^{\theta v}v^{2H}\,dv.
\end{align*}
Substituting the latter value  into the above formula \eqref{var} for $\var X_t$,   we get
\begin{align*}
\var X_t&=
e^{\theta t}t^{2H}
-\frac\theta2\int_0^te^{\theta s}s^{2H}\,ds
+\frac\theta2 e^{2\theta t}\int_0^te^{-\theta s}s^{2H}\,ds\\
&=e^{\theta t}t^{2H}
-\frac12\int_0^ts^{2H}\,d\left(e^{\theta s}+e^{2\theta t-\theta s}\right)
=H\int_0^ts^{2H-1}\left(e^{\theta s}+e^{2\theta t-\theta s}\right)ds.\qedhere
\end{align*}
\end{proof}

Let us investigate the asymptotical behavior of the function $v(\theta,t)$, as $t\to\infty$.
\begin{lemma}\label{l:asymp_v}
\begin{enumerate}[(i)]
\item If $\theta>0$, then
$v(\theta,t)\sim\frac{H\Gamma(2H)}{\theta^{2H}}e^{2\theta t}$, as $t\to\infty$.
\item If $\theta<0$, then
$v(\theta,t)\to\frac{H\Gamma(2H)}{(-\theta)^{2H}}$, as $t\to\infty$.
\item $v(0,t)=t^{2H}$, $t\ge0$.
\end{enumerate}
\end{lemma}

\begin{proof}
$(i)$
If $\theta>0$, then by formula~\eqref{eq:v},
\[
\frac{v(\theta,t)}{e^{2\theta t}}
=He^{-2\theta t}\int_0^ts^{2H-1}e^{\theta s}ds
+H\int_0^ts^{2H-1}e^{-\theta s}ds
\to\frac{H\Gamma(2H)}{\theta^{2H}},
\quad\text{as } t\to\infty.
\]

$(ii)$
Note that
$v(\theta,t)=e^{2\theta t}v(-\theta,t)$,
by~\eqref{eq:v}.
Then the convergence follows from $(i)$.

$(iii)$
The statement follows directly from~\eqref{eq:v}.
\end{proof}

\begin{remark}
For the case $H\in[1/2,1)$ the results of Lemma~\ref{l:asymp_v}  are well-known, see \cite{HN} for $\theta<0$ and \cite{BESO} for $\theta>0$.
\end{remark}

\subsection{Almost sure  limits and bounds for the fractional Ornstein--Uhlenbeck process}

\begin{lemma}[\cite{KMM,kumirase}]
There exists a nonnegative random variable $\zeta$ such that
for all $s>0$, the following inequalities hold true:
\begin{equation}\label{eq:est_KMM}
\sup_{0\le s\le t}\abs{B^H_s}
\le\left(1+t^H \log^2 t\right)\zeta,
\end{equation}
and for $\theta\le0$
\begin{equation}\label{eq:est_KMRS_neg}
\begin{gathered}
\sup_{0\leq s\leq t} |X_s|\leq \left(1+t^H\log^2t\right)\zeta.
\end{gathered}
\end{equation}
Moreover, $\zeta$ has the following property:
there exists $C>0$ such that
$\E\exp\{x\zeta^2\}<\infty$, for any $0<x<C$.
\end{lemma}
\begin{proof}
The bound~\eqref{eq:est_KMM} was established in \cite{KMM}, see also Eq.~(14) in~\cite{kumirase}. It also implies that the inequality~\eqref{eq:est_KMRS_neg} holds for $\theta=0$, since in this case $X_t=x_0+B^H_t$.
The bound~\eqref{eq:est_KMRS_neg} for $\theta<0$ was obtained in~\cite[ Eq.~(19)]{kumirase}.
\end{proof}

\begin{lemma}\label{l:as_non-erg}
For $\theta>0$
\[
e^{-\theta t}X_t\to\xi_\theta
\quad\text{a.\,s., as }t\to\infty,
\]
where
$\xi_\theta=x_0+\theta\int_0^\infty
e^{-\theta s}B^H_s\,ds
\simeq\mathcal N\left(x_0,\frac{H\Gamma(2H)}{\theta^{2H}}\right)$.
\end{lemma}
\begin{proof}
Note that~\eqref{eq:est_KMM} implies the a.\,s.\ convergence
$e^{-\theta t}B^H_t\to0$, as $t\to \infty$.
Therefore, by~\eqref{eq:solution},
\[
e^{-\theta t}X_t=x_0+\theta\int_0^t e^{-\theta s}B^H_s\,ds+e^{-\theta t}B^H_t
\to x_0+\theta\int_0^\infty e^{-\theta s}B^H_s\,ds
\]
a.\,s., as $t\to\infty$.
It follows from Lemmas~\ref{l:distr} and \ref{l:asymp_v} that the limit has the distribution
$\mathcal N\left(x_0,\frac{H\Gamma(2H)}{\theta^{2H}}\right)$.
\end{proof}

\begin{lemma}\label{l:as_erg}
For $\theta<0$
\[
\frac1T\int_0^TX_t^2\,dt\to\frac{H\Gamma(2H)}{(-\theta)^{2H}},
\]
as $T\to\infty$ a.\,s.\ and in $L^2$.
\end{lemma}
\begin{proof}
It was proved in~\cite{CKM} that in this case the process
$Y_t=\int_{-\infty}^te^{\theta(t-s)}\,dB^H_s$
is Gaussian, stationary, and ergodic.
The integral with respect to the fractional Brownian motion here exists as a path-wise Riemann-Stieltjes integral, and can be calculated using integration by parts, see~\cite[Prop.~A.1]{CKM}. It follows from the ergodic theorem that
\[
\frac1T\int_0^TY_t^2\,dt\to\E Y_0^2,
\]
as $T\to\infty$ a.\,s.\ and in $L^2$.
The process $X_t$ can be expressed as
$X_t=Y_t-e^{\theta t}\eta_\theta$,
where
\[
\eta_\theta=\theta\int_{-\infty}^0 e^{-\theta s}B^H_s\,ds-x_0
\]
is a normal random variable.
Using this representation, it is easy  to show that
\[
\lim_{T\to\infty}\frac1T\int_0^TX_t^2\,dt
=\lim_{T\to\infty}\frac1T\int_0^TY_t^2\,dt
=\E Y_0^2.
\]
The value of the limit can be calculated applying Lemmas~\ref{l:distr} and \ref{l:asymp_v}.
Indeed,
\begin{align*}
\E Y_0^2&=\E\left(\int_{-\infty}^0e^{-\theta s}\,dB^H_s\right)^2
=\lim_{t\to-\infty}\E\left(-e^{-\theta t}B^H_{t}+\theta\int_{t}^0e^{-\theta s}B^H_s\,ds\right)^2\\
&=\lim_{t\to\infty}\E\left(-e^{\theta t}B^H_{-t}+\theta\int_{0}^te^{\theta s}B^H_{-s}\,ds\right)^2\\
&=\lim_{t\to\infty}e^{2\theta t}\left(\E\left(B^H_{-t}\right)^2
-2\theta e^{-\theta t}\int_{0}^te^{\theta s}\E B^H_{-t}B^H_{-s}\,ds\right.\\
&\quad+\left.\theta^2e^{-2\theta t}\int_0^t\int_0^te^{\theta(s+u)}\E B^H_{-s}B^H_{-u}\,ds\,du\right)\\
&=\lim_{t\to\infty}e^{2\theta t}v(-\theta,t)
=\frac{H\Gamma(2H)}{(-\theta)^{2H}}.\qedhere
\end{align*}
\end{proof}

\begin{lemma}\label{l:moments}
Let $\theta<0$. Then for any $p\ge1$, there exist positive constants $c_p$ and $C_p$ such that
\begin{gather}
\E\abs{X_t}^p\le c_p \quad \text{for }t\ge0,\label{eq:bound1}\\
\E\abs{X_t-X_s}^p\le C_p\abs{t-s}^{pH} \quad \text{for }\abs{t-s}\le1.\label{eq:bound2}
\end{gather}
\end{lemma}
\begin{proof}
By~Lemmas~\ref{l:distr} and~\ref{l:asymp_v}~$(ii)$, $X_t$ is a normal random variable,
\begin{gather*}
\abs{\E X_t}=\abs{x_0}e^{\theta t}\le\abs{x_0}<\infty,\\
\var X_t\to\frac{H\Gamma(2H)}{(-\theta)^{2H}}<\infty,
\end{gather*}
whence \eqref{eq:bound1} follows.

Assume that $t\ge s\ge0$ and $t-s\le1$.
Let us show that
\begin{equation}\label{eq:bound3}
\E\abs{X_t-X_s}^2\le C_2\abs{t-s}^{2H},
\end{equation}
where $C_2$ is a positive constant.
By~\eqref{eq:sde},
\[
\abs{X_t-X_s}\le\abs{\theta}\int_s^t\abs{X_u}du+\abs{B^H_t-B^H_s}.
\]
Therefore, using \eqref{eq:bound1}, we get
\begin{align*}
\E(X_t-X_s)^2
&\le2\abs{\theta}^2\E\left(\int_s^t\abs{X_u}du\right)^2+2\E\left(B^H_t-B^H_s\right)^2\\
&\le2\abs{\theta}^2(t-s)\int_s^t\E\abs{X_u}^2du+2\left(t-s\right)^{2H}\\
&\le2\abs{\theta}^2c_2^2(t-s)^2+2(t-s)^{2H}
\le2\left(\abs{\theta}^2c_2^2+1\right)(t-s)^{2H}.
\end{align*}
Thus, \eqref{eq:bound3} is proved. Since $X_t-X_s$ has a normal distribution, \eqref{eq:bound2} follows from \eqref{eq:bound3} in the standard way.
\end{proof}

\section*{References}
\bibliographystyle{model1b-num-names}
\bibliography{hypothOU}

\end{document}